\documentclass[12pt]{article}
\usepackage{latexsym,amssymb, amsthm, epsfig}
\usepackage{helvet}         
\usepackage{courier}        

\usepackage{color}
\usepackage{graphicx}
\usepackage{amsmath}
\usepackage{amsfonts}

\usepackage{hyperref}

\newtheorem{theorem}{\bf Theorem}[section]
\newtheorem{lemma}[theorem]{\bf Lemma}
\newtheorem{proposition}[theorem]{\bf Proposition}
\newtheorem{remark}[theorem]{\bf Remark}
\newtheorem{definition}[theorem]{\bf Definition}
\newtheorem{corollary}[theorem]{\bf Corollary}

\newcounter{for}[section]
\newcommand{\be}[1]{\addtocounter{for}{1} \begin{equation}\label{#1}}
\newcommand{\ee}{\end{equation}}

\def\m{{\bf m}}
\def\E{{E}}

\def\P{{\mathbb P}}
\def\R{{\mathbb R}}
\def\Z{{\mathbb Z}}

\def\N{{\mathbb N}}
\def\cL{{\mathcal{L}}}

\def\ka6{6}

\def\ve{{\varepsilon}}

\def\({{\Bigl(}}
\def\){{\Bigr)}}

\def\one{{\mathbf 1}}

\def\square{\ifmmode\sqr\else{$\sqr$}\fi}
\def\sqr{\vcenter{
         \hrule height.1mm
         \hbox{\vrule width.1mm height2.2mm\kern2.18mm\vrule width.1mm}
         \hrule height.1mm}}                  

\def\o{\omega}

\newcommand {\pare}[1] {\left( {#1} \right)}

\theoremstyle{plain}

\theoremstyle{remark}

\def\epr{\end{proof}}

\def\bpr{\begin{proof}}

\def \beq {\begin{eqnarray}}
\def \eeq {\end{eqnarray}}
\def \beqn {\begin{eqnarray*}}
\def \eeqn {\end{eqnarray*}}

\newcommand{\bl}[1]{\begin{lemma}\label{#1}}
\newcommand{\br}[1]{\begin{remark}\label{#1}}
\newcommand{\brs}[1]{\begin{remarks}\label{#1}}
\newcommand{\bt}[1]{\begin{theorem}\label{#1}}
\newcommand{\bd}[1]{\begin{definition}\label{#1}}
\newcommand{\bp}[1]{\begin{proposition}\label{#1}}
\newcommand{\bc}[1]{\begin{corollary}\label{#1}}
\newcommand{\bfact}[1]{\begin{fact}\label{#1}}
\newcommand{\bex}[1]{\begin{example}\label{#1}}
\newcommand{\ec}{\end{corollary}}
\newcommand{\efact}{\end{fact}}
\newcommand{\eex}{\end{example}}
\newcommand{\el}{\end{lemma}}
\newcommand{\er}{\end{remark}}
\newcommand{\ers}{\end{remarks}}
\newcommand{\et}{\end{theorem}}
\newcommand{\ed}{\end{definition}}

\newcommand{\ep}{\end{proposition}}

\newcommand{\bcl}[1]{\begin{claim}\label{#1}}
\newcommand{\ecl}{\end{claim}}

\newcommand{\ecs}{\end{corollary}}
\newcommand{\eers}{\end{exercise}}
\newcommand{\eexs}{\end{example}}
\newcommand{\eems}{\end{example}}
\newcommand{\els}{\end{lemma}}
\newcommand{\eles}{\end{lemmaex}}
\newcommand{\ets}{\end{theorem}}
\newcommand{\eds}{\end{definition}}
\newcommand{\eps}{\end{proposition}}

\newcommand{\bi}{\begin{itemize}}
\newcommand{\ei}{\end{itemize}}
\newcommand{\ben}{\begin{enumerate}}
\newcommand{\een}{\end{enumerate}}

\def\vbar{\mathchoice{\vrule height6.3ptdepth-.5ptwidth.8pt\kern-.8pt}
   {\vrule height6.3ptdepth-.5ptwidth.8pt\kern-.8pt}
   {\vrule height4.1ptdepth-.35ptwidth.6pt\kern-.6pt}
   {\vrule height3.1ptdepth-.25ptwidth.5pt\kern-.5pt}}
\def\fudge{\mathchoice{}{}{\mkern.5mu}{\mkern.8mu}}
\def\bbc#1#2{{\rm \mkern#2mu\vbar\mkern-#2mu#1}}
\def\bbb#1{{\rm I\mkern-3.5mu #1}}
\def\bba#1#2{{\rm #1\mkern-#2mu\fudge #1}}
\def\bb#1{{\count4=`#1 \advance\count4by-64 \ifcase\count4\or\bba A{11.5}\or
   \bbb B\or\bbc C{5}\or\bbb D\or\bbb E\or\bbb F \or\bbc G{5}\or\bbb H\or
   \bbb I\or\bbc J{3}\or\bbb K\or\bbb L \or\bbb M\or\bbb N\or\bbc O{5} \or
   \bbb P\or\bbc Q{5}\or\bbb R\or\bbc S{4.2}\or\bba T{10.5}\or\bbc U{5}\or
   \bba V{12}\or\bba W{16.5}\or\bba X{11}\or\bba Y{11.7}\or\bba Z{7.5}\fi}}

\def \o {\omega}

\def \D {{\bf {\rm {D}}}}

\def\sqr#1#2{{\vcenter{\vbox{\hrule height .#2pt
                             \hbox{\vrule width .#2pt height#1pt \kern#1pt
                                   \vrule width .#2pt}
                             \hrule height .#2pt}}}}
\def\square{\mathchoice\sqr54\sqr54\sqr{4.1}3\sqr{3.5}3}
\def\pmb#1{\setbox0=\hbox{#1}%
   \kern-.025em\copy0\kern-\wd0
   \kern.05em\copy0\kern-\wd0
   \kern-.025em\raise.0433em\box0 }
\def\sqr#1#2{{\vcenter{\vbox{\hrule height.#2pt
     \hbox{\vrule width.#2pt height#1pt \kern#1pt
   \vrule width.#2pt}\hrule height.#2pt}}}}

\def\N{{\mathbb N}}   
\def\Z{{\mathbb Z}}
\def\R{{\mathbb R}}

\def\P{{\mathbb P}}

\def\cL{{\mathcal L}}

\begin{document}
\title{Simulation of quasi-stationary
distributions on countable spaces}

\author{Pablo Groisman\thanks{Departamento de Matem\'atica, Fac. Cs. Exactas y
Naturales, Universidad de Buenos Aires and IMAS-CONICET. {\tt
pgroisma@dm.uba.ar}, {\tt http://mate.dm.uba.ar/$\sim$pgroisma.}}
  \ and  Matthieu Jonckheere\thanks{IMAS-CONICET. {\tt mjonckhe@dm.uba.ar},
 {\tt http://matthieujonckheere.blogspot.com}.
} }

\date{}
\pagestyle{myheadings}
\markright{Simulation of QSD.}

\maketitle



{\abstract Quasi-stationary distributions (QSD) have been widely studied since
the pioneering work of Kolmogorov (1938), Yaglom (1947) and Sevastyanov (1951). They appear as a natural object when
considering Markov processes that are certainly absorbed since they are
invariant for the evolution of the distribution of the process conditioned on
not being absorbed. 
They hence appropriately describe the state of the process at large
times for non absorbed paths. 
Unlike invariant distributions for Markov processes, QSD are solutions of a non-linear equation 
and there can be 0, 1 or an infinity of them. Also, they cannot be
obtained as Ces\`aro limits of Markovian dynamics. These facts make the
computation of QSDs a nontrivial matter. We review different
approximation methods for QSD that are useful for simulation purposes, mainly focused on Fleming-Viot dynamics. We also give some alternative proofs and
extensions of known results.}

\



\section{Introduction}

Let $Q = (q(x,y),\, x,y \in \Lambda \cup \{0\})$ be the rates matrix of a
non-exploding continuous time Markov process $Z(\cdot, \mu)$, with
initial condition $\mu$ on a
countable state space $\Lambda \cup \{0\}$. We assume that $Q$ is irreducible
in $\Lambda$ and think of $0$ as an absorbing state. We also assume that 
absorption is certain.  As usual we denote
$q(x,x) = -\sum_{y\ne x} q(x,y)$.

Let $\tau^\mu$ be the absorption time defined by $\tau^\mu=\inf\{t>0 \colon
Z(t,\mu)=0\}$ and $T_t \mu$
the evolution of the process $Z(\cdot, \mu)$ conditioned on survival
$$T_t \mu (x):= P(Z(t,\mu) = x\ |\ \tau^\mu>t) = \frac{\sum_{z\in\Lambda}
\mu(z)P_t(z,x)}{1-\sum_{z\in\Lambda}
\mu(z)P_t(z,0)}.$$
Here $P_t=e^{tQ}$ is the semi-group associated with the transitions matrix $Q$.
For $\mu$ such that $\sum_z \mu(z)|q(z,z)| < \infty$, using the Kolmogorov
equations for $Z$ we obtain the Kolmogorov equations for the conditioned
evolution
\be{kolmogorov.cond.evolution}
\begin{split}
 \frac{d}{dt}T_t\mu(x) = & \sum_{y\in \Lambda}
 q(y,x)\,T_t\mu(y)+
 \sum_{y\in \Lambda} q(y,0)\,T_t\mu(y)\,T_t\mu(x), \qquad x \in \Lambda,\\
T_0(\mu)(x)= & \mu(x), \qquad x \in \Lambda.
\end{split}
\ee
In particular, we have the semi-group property $T_{s+t}=T_s(T_t)$. We call
$(T_t\mu)$ the {\em conditioned evolution}.  We can hence deduce from (\ref{kolmogorov.cond.evolution}) that any limit point
of $T_t\mu$ is a fixed point of $\{T_s,\ s\ge 0\}$.

\begin{definition}
For a probability measure $\nu$ in $\Lambda$, we  say that it is a quasi-stationary
distribution (QSD) if and only if for any $t \ge 0$ we have $T_t\nu = \nu$.
\end{definition}

Evaluating (\ref{kolmogorov.cond.evolution}) at $t=0$, we obtain a
non-linear equation for $\nu$, i.e. $\nu$ is a QSD if and only if it verifies
\be{intro-1}
0 = \sum_{y\in \Lambda} q(y,x)\,\nu(y)+
 \sum_{y\in \Lambda} q(y,0)\,\nu(y)\,\nu(x).
\ee
Now, if $\nu$ is a QSD, it follows from the Markov property that there exists
$\theta_\nu>0$ such that
$$P(\tau^\nu >t)= \exp(-\theta_{\nu}  t).$$
Hence, a necessary condition
for the existence of a QSD is the existence of a state $x \in \Lambda$ and
$\theta>0$ such that $E(e^{\theta\tau^x})<\infty$. Here we slightly abuse
notation and write $\tau^x$ instead of $\tau^{\delta_x}$. Ferrari et. al (see below)
showed in \cite{FKMP} that if $\Lambda=\N$ and for every $t>0$, $\lim_{x\to
\infty}P(\tau^x\le t)=0$, then this condition is also sufficient for the
existence of a QSD (but does not imply uniqueness). Unfortunately, the additional
condition $\lim_{x\to
\infty}P(\tau^x\le t)=0$ is not just technical (see \cite[p.9]{VM} for a simple
counter-example). As a consequence, an absorbing
process, even with integrable absorption times, might well have no  QSD. It was also
shown \cite{FMP,seneta-veres, VD} that there are a lot of cases where
an infinite number of QSD do exist. These include for instance birth and death process with
constant or linear rates and sub-critical branching processes.

We now describe the dynamics of a process that
is the cornerstone of several methods of both simulation and theoretical analysis
of QSDs.

\paragraph{The $\mu-$return process.}

Given a probability $\mu$ in $\Lambda$, consider a Markov process $Z^\mu$ with the following dynamics:
it evolves according to the dynamics of $Z$ until it hits $0$, at which point,
it instantaneously takes a new position in $\Lambda$ with distribution $\mu$.
The transitions of this process are then given for $x \neq y$ by:
\be{MarkovQSD}
q^\mu(x,y)=q(x,y) + q(x,0) \mu(y).
\ee

This process appears in the literature since the very beginning of the study of
QSD (\cite{bartlett, DSdt}) but the approach was first discarded, arguing that there
is a very weak relation between the QSD of $Q$ and the invariant distribution
of $Z^\mu$ and that this last one can be changed into almost any distribution by a
suitable choice of $\mu$. 

In \cite{FKMP}, the authors considered the
function that maps $\mu \mapsto\Phi(\mu)$, where $\Phi(\mu)$ is the invariant
distribution of $Z^\mu$. In that work, they realized that it was fruitful to
study iterations of this map and they showed that $\nu$ is a QSD for $Q$ if and
only if it
is a fixed point of $\Phi$. More importantly, they obtain the following
theorem.

\begin{theorem}[Ferrari, Kesten, Mart\'inez, Picco, \cite{FKMP}]
Assume $\Lambda=\N$ and that $P(\tau^x >t) \to 1$ as $x \to \infty$ for
every $t>0$. Then a necessary and sufficient condition for the existence of a
QSD is 
\[
 E(e^{\theta \tau^x}) < \infty, \qquad \text{for some } x \in \N, \,\, \theta >0.
\]
\end{theorem}
Their proof strategy is based on the fact that iterations of the map
$\Phi^n(\mu)$ are expected to converge to a QSD. 

Similar approaches were then exploited by different authors and with different goals.


\begin{itemize}
\item The analysis of the dynamics of $Z^\mu$ is the key elements for the proof of
existence and uniqueness of QSD under Doeblin conditions given by Jacka and Roberts \cite{JR95}.

\item Asselah and Castell \cite{AC} use the functional $\Phi$ to prove existence of
QSD for asymmetric conservative attractive particle systems in $\Z^d$. They work
with Ces\`aro means of the iterations $\Phi^n(\mu)$ rather than with the
sequence itself. They prove that limit points of the Ces\`aro means are QSD and
they also used this fact to prove properties (regularity) of the QSD relying on
a priori bounds.

\item Aldous, Flannery and Palacios use also this idea to construct a method to
simulate the QSD in finite spaces \cite{AFP}. We give more details in Section
\ref{sec.om}.

\item Fleming-Viot dynamics are also linked to the dynamics of $Z^\mu$. In fact, the distribution of a tagged particle of a stationary Fleming-Viot system with $N$ particles follows the dynamics $Z^\mu$, with $\mu$ being the empirical distribution of the process (in this case $\mu$ depends on time and is random and correlated with the tagged particle). When FV is in equilibrium, the tagged particle is expected to converge to $Z^{\nu}$ when $N$ goes to infinity, where $\nu$ is a QSD. In Section \ref{sec.tagged} we prove this fact for some models.

\item Also \cite{A, BP} deal with this map to approximate QSDs.
\end{itemize}

As it can be seen, the $\mu$-return process has motivated a lot of work in the
field of QSDs in general and in particular in the problem of their simulation, which is non trivial
 since they can not be simulated directly, using  for example the
rejection method.
In the sequel, we first review some results obtained in that direction in
recent years. We mainly focus on
Fleming-Viot (FV) dynamics, described in Section
\ref{flemingviot},
as they constitute not only a method to practically
simulate the QSD and a theoretical tool to study QSDs, but also an interesting particle system raising many stimulating questions.

We then provide a new proof for the convergence of Fleming-Viot dynamics
to the conditioned evolution in compact time intervals and study the limiting
behavior of a FV tagged particle. 
In Section \ref{sec.om}, we explain the principles of two alternative methods for the computation
of the QSD.
Section \ref{relatedproblems} deals with related open problems.

An updated bibliography on QSDs
can be found at 

\url{http://www.maths.uq.edu.au/~pkp/papers/qsds/qsds.pdf}.

%

\section{Fleming-Viot dynamics}

\label{flemingviot}
We define the {\em Fleming-Viot} process driven by $Z$ (or $Q$) as
follows. It consists of $N$ particles moving independently of the others as a
continuous time Markov process with rates $Q$, until
one attempts to jump to state 0. At that time, it immediately jumps
back to $\Lambda$ by choosing the position of one of the other particles chosen
uniformly at random and starts afresh from this configuration.
This process is Markov with state space $\Lambda^N$ (here $N$ is the
number of particles) and its state at time $t\ge 0$ is denoted $\xi(t,\cdot)$.
The variable $\xi(t,i)$ denotes the position of the $i$-th particle at time $t$.
The generator $\cL^N$ acts on functions $f:\Lambda^{N}\to\R$ as follows
\be{generator}
\cL^N f(\xi)= \sum_{i=1}^N \sum_{x\in\Lambda}
\Bigl[q(\xi(i),x) + q(\xi(i),0)\,
\frac{\sum_{j\not=i}^N \one_{\{\xi(j)=x\}}}{N-1}\Bigr]
\big(f(\xi^{i,x}) - f(\xi)\big),
\ee
where $\xi^{i,x}(i)=x$, and for $j\not=i$, $\xi^{i,x}(j) = \xi(j)$. We use $\{S^N_t, t\ge0\}$ for the semi-group of FV, i.e. 
\[
 S^N_t f (\xi) :=  E(f(\xi(t))|\xi(0)=\xi)
\]

The well definiteness of the process (given that the driving process is well defined) is immediate. Nevertheless we refer the reader to \cite{ afg, afgj,FM} or Section \ref{sec.tagged} for (several) graphical constructions that are useful to infer different properties of the process.

These dynamics were originally introduced by Burdzy, Ingemar, Holyst and March
in the context of Brownian motion in a bounded domain in \cite{BIHM}, where
it was conjectured that the empirical measure should converge to a deterministic
solution of a PDE that represents the distribution of a Brownian particle
conditioned on non-absorption. 

As it is described here, the process is a variation of the original measure-valued process introduced by Fleming
and Viot \cite{fv} (see also
\cite{E}). Many alternative dynamics have been later proposed and studied. Their
common feature is that all the particles evolve independently according to a
common driving process until a {\it selection mechanism} takes place, replacing
particles having specific positions or potentials. In Moran systems for example,
particles branch according to a state-dependent rate and are replaced according
to a fitness functional of the positions. Brunet, Derrida and coauthors
\cite{brunet,
brunet2, brunet-derrida2, brunet-derrida} on the other hand, introduced a system where particles evolve driven by
one dimensional Brownian Motion and branch with a fixed rate, at which time the
left-most particle is instantaneously eliminated. Although these systems are simpler to analyze than FV (in fact much finner results have been obtained for this model \cite{berestycki,maillard}), they can only be defined for random walks in $\Z_+$ as driving process.

Such models appear actually at the corner of several scientific fields.
They were originally introduced in genetics to study the propagation of genes in
a population.
In a completely different context, they may represent the evolution of particles
in a medium with an obstacle, and are linked with traveling waves equations in
physics \cite{brunet}. For more details, refer to \cite{delmoral} and the
references therein.

Coming back to FV, a generic question is hence to understand for a given driving Markov process,
how the empirical density of the FV process at fixed times or under its
stationary regime (if it exists) approximates the conditioned
evolution or the QSD of the driving process respectively.
This line of research was pursued by several authors and the conjecture was
settled for many cases of interest where the driving process {\it has a unique
QSD}, like for jump processes with finite state spaces \cite{afg}, chains with
Doeblin conditions \cite{FM} and diffusion processes on bounded
domains \cite{bieniek, BHM, GK, GK1}. 

An important and challenging problem then consists in studying the FV process
when the driving process has an infinite number of QSDs. This question was
recently answered for the case of a driving process given by sub-critical
branching. In  \cite{afgj} it is shown that FV driven by this process is ergodic
and that the empirical distribution of particles in equilibrium converges to the
{\em minimal QSD} in the following sense.
\begin{definition}
A  QSD $\nu^*$ is called  {\it minimal} if
\[
E(\tau^{\nu^*}) = \inf \{E(\tau^{\nu}) \colon \nu \text{ verifies
\eqref{intro-1}}\}.
\]
\end{definition}
This phenomenon is called {\em selection principle}, in the sense that FV ``selects'' the minimal QSD among all the QSD. This fact also gives evidence in favor of the minimal QSD to be interpreted as the one that represents the (quasi) stationary behavior of the process before absorption takes place.

Denote for $x\in \Lambda$ and $\xi\in\Lambda^N$
\be{empiric}
m(x,\xi) := \frac1N\sum_{i=1}^N \one\{\xi(i)=x\} \ee the proportion
of $\xi$-particles at position~$x$. We also think of $m(\cdot, \xi)$ as a
probability measure in $\Lambda$. With a slight abuse of notation, for a measure $\gamma^N$ in $\Lambda^N$ we write $m(x,\gamma^N)$ for the random measure $m(x,\xi)$ where $\xi$ is a random configuration of particles with distribution $\gamma^N$.  
Applying the generator of FV to the empirical distribution, we get
\be{generator-m} {\cal L}^N m(y,\cdot)(\xi)= 
\sum_{x\in\Lambda}  q(x,y) m(x,\xi) + {\frac{N}{N-1}} \sum_{x\in\Lambda} q(x,0)\,m(x,\xi)
m(y,\xi),
 \ee
which looks very similar to \eqref{kolmogorov.cond.evolution}.
This provides another convincing heuristic
for the approximation of the conditioned evolution and QSD by FV.

It may be enlightening at this point to emphasize one of the main difficulties
encountered when
studying Fleming-Viot dynamics. Contrary to all the other mentioned genetic
interacting jumps particle models, the FV process is not attractive: 
natural partial orders are not preserved by trajectories of the processes
started from ordered initial positions. This complicates drastically the analysis and impedes the
construction of simple coupling and stochastic comparisons for proving
convergence properties. Another
consistent difficulty arises because of the stochastic
dependence between the times at which the selection mechanism takes place and
the trajectories of the particles. As a consequence,  the well definiteness of the process for all times when it is driven by diffusions is a difficult problem\footnote{ In countable spaces however,
this follows immediately but this fact illustrates the possible complications
that can arise with this dynamics.
} \cite{bieniek,  BBS, GK}.
These difficulties can also be thought as somehow linked to the critical nature of those dynamics where the population size is kept constant. 

In the rest of this section, we aim at discussing:
\begin{enumerate}
\item the convergence of the empirical distribution for fixed time intervals,

\item the ergodicity and convergence of the empirical distribution under the
stationary measure of FV,

\item the behavior of a tagged particle.

\end{enumerate}
We provide a new proof of convergence with order $1/\sqrt{N}$ for the first
item and establish via a simple coupling argument the limiting distribution of a tagged
particle. We also review the recent results obtained for long-time behaviors of the
empirical distribution.


\subsection{Convergence to the conditioned evolution}\label{sec:condE}

In this section, we investigate the convergence of the empirical distribution for fixed time intervals. 
In \cite{V}, Villemonais proves a general result of convergence to the conditioned evolution with a rate of order $1/\sqrt{N}$. The proof in \cite{V} is based on martingale arguments under the assumption that the number of jumps of the FV process does not explode in compact time intervals.
We here propose a new proof (leading to the same speed of convergence) working
under quite general technical conditions on the transition rates and based
on different techniques. It uses in particular the method of supersolutions.




\begin{theorem}
\label{thm-conv}Assume $\sum_{y\in \Lambda} q(y,x) \le C$ for all $x\in
\Lambda \cup \{0\}$. If $-\sum_z \mu(z)q(z,z) <\infty$, then 
\be{conv-cond-ev2}
\sup_{0\le t \le T}    E \left|  m(x,\xi(t)) -
T_t\mu(x)) \right |  \le e^{5C T} \left(\frac{3}{\sqrt N}
+ \|\mu - m(\xi(0))\|_{TV} \right)
\ee
\end{theorem}
\begin{remark}
Note that the technical condition $\sum_{y\in \Lambda} q(y,x) \le C$ is valid
for a wide class of processes. In particular, we do not assume that rates are
uniformly bounded nor the condition $\sum_{y\in \Lambda} q(x,y) \le C$.
\end{remark}

The proof of Theorem \ref{thm-conv} is based on the following propositions. The
first one is a useful decorrelation result proved in \cite{afg} and the second
one proves the closeness between the semigroups associated with the conditional
evolution of the driving process and the dynamics of the empirical measure of
the FV process.

\begin{proposition}[Proposition 2 of \cite{afg}]\label{prop-chaos}
Let $C_0:= \sup_{x \in \Lambda} q(x,0)$. For each $t>0$, and any
$x,y\in\Lambda$ 
\be{correlations}
\sup_{\xi\in\Lambda^N}\big|S_t^N (m(x,\cdot)m(y,\cdot))(\xi)-
S_t^Nm(x,\cdot)(\xi)\,S_t^N m(y,\cdot)(\xi)]\big|\le \frac{2e^{2C_0t}}{N}.
\ee
\end{proposition}
\begin{remark}
 Proposition \ref{prop-chaos} is proved in \cite{afg} in fact for processes
with bounded rates, but the proof can be extended to this case with no
difficulty.
\end{remark}

\begin{proposition} \label{prop-semigroups} Assume that $\sum_{y\in \Lambda} q(y,x) \le C$ for all $x\in \Lambda\cup\{0\}$. If $-\sum_z \mu(z)q(z,z) <\infty$, then 
\be{conv-cond-semigroups}
\sup_{0\le t \le T}   \left | E m(x,\xi(t)) -
T_t\mu(x) \right |  \le e^{5CT} \left(\frac{1}{N} +  \|\mu -
m(\xi(0))\|_{TV} \right)
\ee

\end{proposition}

Proposition \ref{prop-semigroups} is based on the following lemma.

\begin{lemma}\label{comparison}
Assume that $\sum_{y\in \Lambda} q(y,x) \le C$ for all $x\in \Lambda\cup\{0\}$. Let 
$e\colon[0,T]\times\Lambda \to \R$ be a function such that for $ x \in \Lambda$, $t \in [0,T]$
\be{super-solution}
\begin{split}
\frac{\partial}{\partial t} e(t,x) & \ge \sum_{z\in \Lambda}q(z,x)
e(t,z)+ c_1(t,x)e(t,x)+ c_2(t,x) \sum_{z \in \Lambda}q(z,0) e(t,z),\\
\end{split}
\ee
and $e(0,x)  \ge  0$ for $x \in \Lambda$, with $c_1$ bounded on $[0,T]$ and $c_2$ bounded and nonnegative. Then
$e(t,x) \ge 0$ for all
$t \in [0,T],$ $x \in \Lambda$.
\end{lemma}

\bpr
First, we can assume without loss of generality that $C + c_1 + c_2 <0$ in
$[0,T]$. If not, we consider $w(t,x):=\exp{(-\lambda t)} e(t,x)$, which verifies
the same inequality with $c_1$ replaced by $c_1-\lambda$. Second, we can also
assume that we have strict inequalities in \eqref{super-solution}. If not we
consider $e^\ve(t,x)=e(t,x) + \ve$ which verifies \eqref{super-solution} with
strict inequalities thanks to $C+c_1+c_2<0$. Since $e^\ve \to 0$ as $\ve \to 0$,
if we prove $e^\ve(t,x) \ge 0$ for all $t,x,\ve$, this implies our conclusion.

So, assuming that the conclusion of the lemma is false, there is a first time
$t^*$ and a state $x^*$ such that $e(t^*,x^*)=0$. At that time we have
\[
0 \ge \frac{\partial}{\partial t}  e(t^*,x^*) > \sum_{z\ne x^*} q(z,x^*) e(t^*,y) +c_2(t^*,x^*)
\sum_{z \in \Lambda}q(z,0) e(t^*,z^*) \ge 0,
\]
a contradiction that concludes the proof.
\epr

\bpr
[{\bf Proof of Proposition \ref{prop-semigroups}}]
Define
\be{symbol-1} 
u(t,x)=S_t^N\pare{m(x,\cdot)}(\xi),\quad
\text{and}\quad v(t,x)=T_t\mu(x)
\ee
and let $\delta(t,x)=u(t,x)-v(t,x)$. Recall that
\be{dynamic-v}
\frac{\partial}{\partial t} v(t,x)=\sum_{z\in \Lambda}q(z,x)
v(t,z)+ \sum_{z \in \Lambda}q(z,0) v(t,z)v(t,x),
\ee
\be{a82} 
\frac{\partial}{\partial t} u(t,x)=\sum_{z\in \Lambda}q(z,x)
u(t,z)+ \sum_{z \in \Lambda}q(z,0) u(t,z)u(t,x)+
R(\xi;x,t),
\ee
where,
\be{def-R}
R(\xi;x,t)= \sum_{z\in\Lambda}q(z,0) \pare{\frac{N}{N-1}S_t^N\pare{m(z,\cdot)m(x,\cdot)}(\xi)
-S_t^Nm(1,\cdot)(\xi)S_t^Nm(x,\cdot)(\xi)}.
\ee
Proposition~\ref{prop-chaos} implies that
\be{step-31}
\sup_{\xi}|R(\xi;x,t)|\le \frac{ 2C_0 e^{2C_0 t}}{N}.
\ee
Then, for $\delta(t,x)$ we get that:
\be{eq-delta}
\frac{\partial}{\partial t} \delta(t,x)=\sum_{z\in \Lambda}q(z,x)
\delta(t,z)+ \sum_{z \in \Lambda}q(z,0) [u(t,z)\delta(t,x)+ v(t,x)\delta(t,z)]
+  R(\xi;x,t).
\ee

Call $C_N:=\frac{1}{N} + \|\mu - m(\xi(0))\|_{TV}$ and consider the {\em
supersolution } given by $\bar \delta (t,x)=\bar \delta(t):=C_N\exp(5C t).$
We have
\be{eq-deltabar}
\begin{split}
\sum_{z\in \Lambda}q(z,x)
\bar \delta(t,z)+ \sum_{z \in \Lambda}q(z,0) [u(t,z)\bar \delta(t,x)+
v(t,x)\bar \delta(t,z)]
+  R(\xi;x,t) =\\
\bar \delta (t) \sum_{z\in \Lambda}q(z,x)
+ \bar \delta(t) \sum_{z \in \Lambda}q(z,0)u(t,z) +
\bar\delta(t) v(t,x) \sum_{z \in \Lambda}q(z,0)
+  R(\xi;x,t) \le \\
C_Ne^{5C t} \left (\sum_{z\in \Lambda}q(z,x)
+  \sum_{z \in \Lambda}q(z,0)u(t,z) +
v(t,x) \sum_{z \in \Lambda}q(z,0) \right) +  \frac{ 2C_0e^{2C_0 t}}{N} \le \\
C_N 5 C e^{5C t} =
\frac{d}{dt}\bar \delta (t,x).
\end{split}
\ee
In the last line we have used that $C_0 \le C$ and $C_N \ge 1/N$. Hence the function
$e(t,x)=\bar \delta(t,x)-\delta(t,x)$ verifies \eqref{super-solution} with
$c_1(t)=\sum_{z \in \Lambda} q(z,0)u(t,z)$ and $c_2(t) =  v(t,x)$. Lemma
\ref{comparison} then implies that 
\[
 \delta (t,x) \le e^{5C t}  \left (\frac{1}{N} + \|\mu -
m(\xi(0))\|_{TV} \right), \qquad x\in \Lambda, \, t\in [0,T].
\]
Proceeding in the same way with $\tilde e(t,x):= \bar \delta(t,x) + \delta(t,x)$ instead of $e$ we obtain the
opposite inequality and hence
\be{second.term}
| \delta (t,x)| \le  e^{5C t}  \left (\frac{1}{N} + \|\mu -
m(\xi(0))\|_{TV} \right), \qquad x\in \Lambda, \, t\in [0,T].
\ee

\epr

We  can now proceed to the proof of the theorem.

\bpr[Proof of Theorem \ref{thm-conv}]

We apply the triangular inequality to get that for some $\kappa>0$ and every $t \in [0,T]$
\[
\begin{split}
  E \left [\left |m(x,\xi(t)) -
T_t\mu(x) \right | \right ] \le  &  E \left [\left |m(x,\xi(t)) -
u(t,x) \right | \right ] +  |u(t,x) -
v(t,x)|.
\end{split}
\]
The control of the first term follows from Proposition \ref{prop-chaos} (with $y=x$)
\[
\begin{split}
E \left [\left |m(x,\xi(t)) -
u(t,x) \right | \right ] & \le   E \left [\left |m(x,\xi(t)) -
u(t,x) \right |^2\right ]^{1/2}\\ 
& = |E(m(x,\xi(t))^2)-u(t,x)^2|^{1/2}\\
& \le \sqrt{2/N} e^{C_0 t}
\end{split}
\]
while the second term follows from \eqref{second.term}.

\epr

\subsection{Long-time behavior}

The motivation for studying FV steams in particular from the exciting possibility of approximation of QSD, which are in general difficult
objects to simulate. The kind of result one expects concerning the long-time behavior are

\begin{enumerate}
 \item For each $N\ge 1$, the Fleming-Viot process is ergodic, 
with invariant measure $\lambda^N.$

\item Let $\nu^*$ be the minimal QSD for the driving process, then
\be{main-1}
\lim_{N\to\infty} \int |
m(x,\xi)- \nu^*(x)| d\lambda^N(\xi)=0.
\ee

\end{enumerate}

\paragraph{A unique QSD.}

For Markov jump processes, these results have been proven in the case of
finite state space in \cite{afg} and for processes satisfying a Doeblin condition in \cite{FM}.
For finite state space, ergodicity is immediate while it follows from a
``perfect simulation'' argument in \cite{FM}.
To prove the convergence of the empirical distribution under the invariant
measure  of FV is a much more involved issue.
In both \cite{afg,FM}, the key element of the proof is a control of
the correlations between particles. 
In the case of a finite state space, the control of the correlations for a fixed
time 
is obtained uniformly on the set of initial distributions (see
\ref{prop-chaos}).
 This allows to extend this control under the stationary measure of FV. 
In \cite{afg}, this step is done via a regeneration argument strongly connected
to the Doeblin condition.


\paragraph{An infinity of QSD: the Galton-Watson case.}

Many processes do actually have an infinity of quasi-stationary measures.
The convergence mentioned in \eqref{main-1} hence means that
FV \emph{selects} the minimal QSD.
Proving these results becomes in this case much more challenging as
one needs to prove not only tightness but 
also that the empirical measure of FV under the stationary measure belongs
to the domain of attraction of the minimal QSD.

An important example of processes owning an infinity of QSD  is the
subcritical branching process. For a full description of the number of QSD for branching processes see \cite{seneta-veres}, and for general birth and death
processes see \cite{C, FMP, KmG,VD}.

Up to our knowledge, the only result that deals with convergence of FV in
equilibrium when the QSD is not unique is the following one, which is proved in \cite{afgj}. 

\begin{theorem}[Asselah, Ferrari, Groisman, Jonckheere,
\cite{afgj}]\label{theo-main}
Consider a subcritical  Galton- Watson process whose offspring law has some
finite positive exponential moment. 
Let $\nu^*$ be the minimal quasi stationary distribution for the
process conditioned on non-extinction.
Then 
\begin{enumerate}
\item For each $N\ge 1$ the associated  $N$-particle Fleming-Viot process
is ergodic with invariant measure $\lambda^N$.
\item 
The empirical measure of FV converges to a deterministic measure which is the minimal QSD of the driving process: 
\be{main-1b}
\lim_{N \to \infty }\int |m(x,\xi) - \nu^*(x)| \, d\lambda^N(\xi) =0.
\ee
\end{enumerate}
\end{theorem}

There are many other open questions concerning the long-time behavior of FV for large $N$. We deal with them in Section \ref{relatedproblems}.

\subsection{The tagged particle limit}\label{sec.tagged}
We now consider a particular particle (say particle $1$) that we call the {\em
tagged particle}. We show that under the invariant measure of FV, the tagged
particle converges to a limiting
process that has the QSD distribution as marginal and is absorbed at exponential
times. At those times, it jumps back instantaneously to $\Lambda$ according to the QSD. A similar result holds out of equilibrium, but in this case the ``return`` jumps are according to the conditioned evolution. More precisely, the limiting process, $(Y(t, \mu), t\ge 0)$ is a continuous time non-homogeneous Markov jump process with initial condition $\mu$ and rates given by
\be{rates.tagged}
q^\mu_t(x,y) = q(x,y)  + q(x,0)T_t\mu(y), \quad x,y \in \Lambda. 
\ee
Observe that if  $\mu$ is a QSD, then $q^\mu_t(x,y)$ is time-independent and also $\mu$ is invariant for \eqref{rates.tagged}. Moreover, in this case \eqref{rates.tagged} are the rates of the $\mu-$ return process described in the introduction.

This problem has been previously addressed by Grigorescu and Kang \cite{gk2} when the driving process is Brownian Motion in a bounded domain.

In order to state our result we need to consider the Skorohod space $\D = \D(\mathbb R^+,\Lambda)$ of right-continuous functions from
$\mathbb R^+ \to \Lambda$ with finite left limits, equipped with the Skorohod topology \cite{B}. We also introduce the notation $\xi^{\gamma}$ for a FV process with generator \eqref{generator} and initial condition distributed according to $\gamma$. Finally, $\xi^\xi$ is used when $\gamma=\delta_\xi$.
\begin{proposition}
\label{prop.tagged}
Assume that $C_0:=\sup_{x\in \Lambda}q(x,0) < \infty$ and let $\gamma^N$ be a sequence of
measures in $\Lambda^N$ such that
\be{cond-ev3}
\lim_{N\to\infty}\sup_{0\le t \le T}    E \big|  m(x,\xi^{\gamma^N}(t)) -
T_t\mu(x) \big|=0 \qquad \mbox{for all } x \in \Lambda.
\ee
 Assume further that $\xi^{\gamma^N}(0,1)$ converges in distribution to a
random variable with distribution $\mu$. Then the tagged particle $\xi^{\gamma^N}(\cdot,1)$
converges in distribution to the process $Y(\cdot,\mu)$ as elements in $\D$.


\end{proposition}

As a consequence of Proposition \ref{prop.tagged} we obtain the following theorem.

\begin{theorem}
\label{cor.tagged}
Assume that for each $N\ge 1$, the  FV process has an invariant measure
$\lambda^N$ and that for every $x$
\be{conv.equi}
\lim_{N\to \infty} \int |m(x,\xi)- \nu(x)| \, d\lambda^N(\xi)= 0. 
\ee
Then, under $\lambda^N$ the tagged particle converges in
distribution to the stationary process $Y(\cdot, \nu)$. 
\end{theorem}

\begin{remark}
Hence Proposition \ref{prop.tagged} holds for every jump
Markov process that verifies \eqref{conv-cond-ev2} if $m(x,\xi^{\gamma^N}) \to \mu$, which include a large class
of absorbing processes as shown by Theorem \ref{thm-conv}.
Theorem \ref{cor.tagged} holds when one assumes in addition convergence of FV in equilibrium, which is a much more restricted class but includes the models
of \cite{afg, afgj,FM}. Of particular interest is the case of the driving process being a sub-critical branching process where an infinite number of QSD arise. In this case the results of \cite{afgj} combined with Theorem \ref{cor.tagged} gives that the tagged particle in equilibrium converges to the stationary $\nu^*-$return process, the probability $\nu^*$ being the minimal QSD.

\end{remark}


\subsubsection{Graphical construction}
\label{graphical.construcion}
The proof follows by coupling the trajectory of the limiting Markov process defined previously
and the trajectory of a tagged particle of  FV. 
To make the proof rigorous, we need to carefully build the coupling using a so called graphical
construction. We first describe the construction for the FV process.

Recall that $C_0=\sup_{x\in \Lambda}q(x,0)$ is the (maximum) 
absorption rate, and let
\[
\bar q := \sup_{x\in\Lambda} \sum_{y\in\Lambda\setminus \{x\}} q(x,y);
\quad p(x,y):=\frac{q(x,y)}{\bar q},\; y\neq x;
\quad p(x,x):=1-\sum_{y\in\Lambda\setminus \{x\}} p(x,y).
\]
We assume for simplicity that $\bar q < \infty$\footnote{this is not needed but otherwise once has to define a marked Poisson process for each potential position of a particle which complicates the exposition.}. To each particle $i$, we associate two independent marked Poisson processes
$(\o_i^I,\o^V_i)$ on $\R_+$, which we call respectively the {\it internal} and {\it voter}
point processes, described as follows.
\begin{itemize}
\item For each $i$, the internal Poisson process has intensity $\bar q$. The marks are i.i.d uniform variables in $[0,1]$. We use $U^k_i$ for the mark of the $k-$th occurrence of $\o_i^I$. If the Poisson process rings at time $t^i_k$  with a  mark $ u^k_i$ whereas particle $i$ is at $x$, then it jumps to $y$ if and only if $u^k_i \in I(x,y)$, where for each $x \in \Lambda$, $\{I(x,y) \colon y\in \Lambda\}$ is any partition of $[0,1]$ composed of Borel sets such that $|I(x,y)|=p(x,y)$. For details see \cite[Chapter 1]{ferrari-galves}.

\item The voter point process is also defined for each $i$. It has intensity $C_0$ and the marks are i.i.d. uniform variables on $[0,1]\times[0,1]$. We denote them by $(B^k_i,V^k_i)$. If $\o_i^V$ rings at time $t_i^k$, with a mark $(b_i^k,v_i^k)$ whereas particle $i$ is at $x$, then it keeps its position if $b_i^k>\frac{q(x,0)}{C_0}$. If not, it jumps to zero and instantaneously comes back to $\Lambda$ by jumping to $y$ if and only if $v_i^k \in J_i(t,y)$. Here $\{J_i(t,y) \colon y \in \Lambda\}$ is a (random) partition of $[0,1]$ by Borel sets such that $|J_i(t,y)|=\frac{1}{N-1}\sum_{j\ne i} \one\{\xi(t,j)=y\}$.

According to this procedure, particle $i$ jumps to $y$ with rate
\[
 \bar q \frac{q(\xi(t,i),y)}{\bar q} + C_0
\frac{q(\xi(t,i),0)}{C_0}\frac{1}{N-1}\sum_{j\ne i} \one\{\xi(t,j)=y\},
\]
which coincides with the rates of \eqref{generator}.
\end{itemize}
We call $\o=((\o_i^I,\o_i^V),i\in \{1,\dots,N\})$ an i.i.d.\,sequence of
marked PP. 


For any $s<t$ real numbers, we denote by $\o_i[s,t)$ and $\o_i[s,t)$ the
projections of the marked-times in the time
period $[s,t)$ and $[s,t]$, respectively.

We construct $\{\xi(t),t\ge 0\}$ in such a way that $\xi(t)$ is a function of the
initial configuration $\xi$ and the time-marks $\o[0,t]$, $t\ge 0$.  Fix an
initial configuration $\xi\in \Lambda^N$, and $t>0$.  There is, almost surely,
a finite number of time-marks within $[0,t]$, say $K$; let $\{s^k,0\le k\le K\}$
be the ordered time-realizations with $s^0=0$. We build $\xi(\cdot)$ inductively as
follows.
\begin{itemize}
\item At time $s^0=0$, the configuration is $\xi(0)=\xi$.
\item Assume $\xi({s^k})$ is known. For $s \in [s^k,s^{k+1})$, set $\xi(s)=
  \xi({s^k})$. We describe now $\xi({s^{k+1}})$.
\begin{itemize}
\item If $s^{k+1}$ corresponds to an internal time of particle $i$ and mark $u$,
  we move particle $i$ to $y$ if and only if $u \in I(\xi({s^k},i),y)$.  

\item If $s^{k+1}$ corresponds to a voter time of particle $i$ and mark
  $(b,v)$, we move particle $i$ to the position $y$ if and only if
  $b\le \frac{q(x,0)}{C_0}$ and $v \in J_i(t,y)$.
\end{itemize}
\end{itemize}
It is straightforward to check that $\{\xi(t),t\ge 0\}$, as constructed above, has generator
given by \eqref{generator}, see \cite{Harris}.

\subsubsection{The limiting process} We now construct jointly with FV the
limiting
process $\{Y(t,\mu), t\ge 0\}$ with initial condition $\mu$ as a deterministic
function of $(\o^I_1, \o^V_1)$ and a uniform random variable $U$. When is not
necessary, we omit to display the dependence on $\mu$ and write just
$Y(\cdot)$. 
Let $\{s^k_1, 0\le k \le \tilde K\}$ be the ordered time realizations of
$(\o^I_1, \o^V_1)$ in $[0,t]$, with $s^0_1=0$.

\begin{itemize}
\item In order to couple this process with FV, we use the same variable $U$
uniformly distributed in $[0,1]$ to define
\[
\xi(0,1)=y \iff U \in J_1(0,y) \quad \mbox{and}\quad Y(0,\mu)=y \iff U \in
\tilde J(0,y).
\]
The sets $J_1(0,y)$ and $\tilde J(0,y)$ are defined below. To get the correct
distribution for the initial configuration of FV ($\gamma^N$) we construct the
configuration of the remaining particles with the conditional distribution
$\gamma_n|_{\xi(0,1)}$ (the procedure used to do this is not relevant). Next we describe the evolution of the limiting process.

\item Assume $Y({s^k_1})$ is known. For $s\in [s^k_1,s^{k+1}_1)$, set $Y(s)=
  Y({s^k_1})$. We describe now $Y({s^{k+1}_1})$.
\begin{itemize} 
\item If $s^{k+1}_1$ corresponds to an internal time of particle $1$ and mark $u$, then $Y$ jumps to $y$ if and only if $u \in I(Y({s^k_1}),y)$.  

\item If $s^{k+1}_1$ corresponds to a voter time of particle $1$ and mark
  $(b,v)$, then $Y$ jumps to $y$ if and only if
  $b\le \frac{q(x,0)}{C_0}$ and $v \in \tilde J(t,y)$.  Here $\{\tilde J(t,y) \colon y \in \Lambda\}$ is a partition of $[0,1]$ by Borel sets such that $|\tilde J(t,y)|=T_t\mu(y)$.
\end{itemize}

\end{itemize}

It is also straight-forward to see that $\{Y(t), t\ge 0\}$ has rates given by \eqref{rates.tagged}.

Observe that if $v_1^k \in \cup_{y \in \Lambda} [J_1(s_1^k,y)\cap
\tilde J(s_1^k,y)]$ for every $0\le k  \le \tilde K$, then $\xi(\cdot,1) = Y(\cdot)$ in
$[0,t]$. So we construct $J_1(\cdot,y)$ and $\tilde J(\cdot,y)$ in order to
maximize their intersection. We consider a numbering of the elements of
$\Lambda$ $(y_l, l \ge 1)$ and define for $t=0$

\[
 J_1(0,y_l)=\left[\sum_{r=1}^{l-1}\mu_N(y_r),
\sum_{r=1}^{l}\mu_N(y_r)\right),
\]
\[
\tilde
J(0,y_l)=\left[\sum_{r=1}^{l-1}\mu(y_{r}),\sum_{r=1}^{l}
\mu(y_r)\right).
\]
Here $\mu_N$ is the distribution of $\xi(0,1)$. For $t>0$ these sets take the
form
\[
 J_1(t,y_l)=\left[\sum_{r=1}^{l-1}\frac{\sum_{i>1}\one\{\xi(t,i)=y_{r}\}}{N-1},\sum_{r=1}^{l}\frac{\sum_{i>1}\one\{
\xi(t,i)=y_r\}}{N-1}\right),
\]
\[
\tilde J(t,y_l)=\left[\sum_{r=1}^{l-1}T_t\mu(y_{r}),\sum_{r=1}^{l}T_t\mu(y_r)\right).
\]
We also define
\[
 \psi(t):=1- \one\{v_1^k \in \cup_{y \in \Lambda} [J(t_1^k,y)\cap \tilde
J(t_1^k,y)] \text{ for every } k \text{ with } t_1^k\le t\}.
\]
Note that $\psi(t)=1$ if the trajectories of both processes do not coincide on $[0,t]$.
We have the following equation for the evolution of $\psi$.
\begin{lemma}\label{lemma.psi}
\be{psi}
\P\pare{\psi(t)=1} \le \P\pare{\psi(0)=1} +  C_0 \int_0^t \E \left[\sum_{x \in
\Lambda}|m(x,\xi^{\gamma^N}(s))-T_s(\mu)(x)| \right]\, ds.
\ee
\end{lemma}

\begin{proof}

Since $\psi(s-)$ is $\sigma([0,s))$-measurable, we have
\be{step-4}
\P\pare{\psi(t)=1}= \P\pare{\psi(0)=1} +
\int_0^t\E \left [ \frac{d}{ds}\P\pare{\psi(s)=1\big|\sigma(\o[0,s))} \right ]ds,
\ee
and
\be{step-5}
\begin{split}
\frac{d}{ds} & \P\pare{\psi(s)= 1| \sigma(\o[0,s))} \\
 = & \sum_{y\in\Lambda} q(y,0)\one\{Y(s-)=y, \psi(s-)=0 \} \sum_{x \in
\Lambda}|m(x,\xi^{\gamma^N}(s-))-T_{s}(\mu)(x)|\\
\le & C_0 \sum_{x \in \Lambda}|m(x,\xi^{\gamma^N}(s-))-T_{s}(\mu)(x)|.
\end{split}
\ee
Observe that this also justifies the interchange of $E$ and $\frac{d}{ds}$ in \eqref{step-4} and completes the proof.

\end{proof}

\begin{proof}[Proof of Proposition \ref{prop.tagged}] 

Recall that to prove convergence in $\D$, it is enough to show convergence in $\D_T=\D([0,T],
\Lambda)$, the Skorohod space of c\`adl\`ag functions defined in $[0,T]$ with values in $\Lambda$, for every $T>0$, cf. \cite[Ch.3, Sec 16]{B}. So, let
$G_T\colon \D_T \to \R$ be a bounded, continuous function and define
\[
A_N(T):=|G_T(\xi(\cdot,1)) - G_T(Y(\cdot,\mu))|. 
\]
Using that when $\psi(t)=0$, the trajectories of $\xi(\cdot,1)$ and
$Y(\cdot,\mu)$
coincide on $[0,t]$, we get that
\[
\begin{split}
 E(A_N(T)) \le  &   E(A_N(T)\one\{\psi(T)=0\}) + E(A_N(T)\one\{\psi(T)=1\})\\
 \le & 2\|G_T\|_\infty P(\psi(T)=1) \\
\le & 2\|G_T\|_\infty \left(\sum_{x\in \Lambda} |\mu_N(x)-\mu(x)| + C_0 T \sup_{0\le t
\le T} \sum_{x
\in
\Lambda}  \E (|m(x,\xi^{\gamma^N}(t))-T_t\mu(x)|)\right).
\end{split}
 \]
In the last line, we bound using Lemma \ref{lemma.psi}. Now, since $\{T_t\mu, 0\le t \le T\}$ is tight (as probability measures in
$\Lambda$) and $m(\cdot,\gamma^N)\to \mu$,
the convergence
\eqref{cond-ev3} implies that $\{ \E (m(x,\xi^{\gamma^N}(t))),  0\le t \le
T, N\ge 1\}$ is also tight and hence, for any $\tilde \ve>0$, there
exists a finite set $\Lambda_F$ such that
\[
\sup_{0\le t \le T}  \sum_{x\in \Lambda_F^c} \mu_N(x) + \mu(x) +  \E (m(x,\xi^{\gamma^N}(t)))+ E(T_t\mu(x)) < \tilde \ve.
\]

Next we use again the pointwise convergence \eqref{cond-ev3} for $x\in
\Lambda_F$ and that $\mu_N \to \mu$ to get
\[
 \limsup_{N\to \infty} E(A_N(T))  \le \tilde \ve.
\]
This concludes the proof.
\end{proof}

\begin{proof}[Proof of Theorem \ref{cor.tagged}] Since for each $N$, the
measure $\lambda^N$ is exchangeable and the empirical measure $m(\cdot,
\lambda^N)$ distributed according to $\lambda^N$ converges to
$\nu$ in probability, we have that the tagged particle $\xi(1)$ in equilibrium
converges in distribution to $\nu$.
So we can apply Proposition \ref{prop.tagged} with $\mu=\nu$.
\end{proof}

\section{Other ways to approximate QSDs}
\label{sec.om}

Since the QSD are eigenvalues of the transition matrices, in finite (small) spaces the computation of the (unique) QSD can be addressed with standard (or not so standard,
but generic) methods. However this can not be done (at least a
priori) when the state space is huge or even countable. In this section we
introduce two alternatives to FV dynamics to simulate the QSD.

\subsection{The method of Aldous-Flannery-Palacios}
This method was proposed in \cite{AFP} to simulate the quasi-stationary
distribution of a finite Markov chain in discrete time. To simplify the
exposition, in this section we consider $Z_n$ as a discrete time Markov chain
in a finite state space $\Lambda \cup \{0\}$ with transition matrix
$\hat P$. We use $\tilde P$ for the restriction of $\hat P$ to $\Lambda$ and we assume
$\tilde P$ is irreducible and aperiodic. Hence, from Perron-Frobenius theorem we have the existence of a unique
quasi-stationary probability measure $\nu$ that verifies
\be{qsd.disc}
 \sum_{x\in \Lambda} \nu(x)\tilde P(x,y) = \lambda \nu(y), \quad y \in \Lambda.
\ee
So, in order to compute $\nu$ one has to solve \eqref{qsd.disc}, but this is
unpractical if $\Lambda$ is large. Also the acceptance-rejection method is not
useful since $P(\tau>n)$ goes to zero exponentially fast.

The method of Aldous-Flannery-Palacios is motivated on the $\mu-$return process described in the introduction, i.e. the renewal process
$Z^\mu_n$ with transition matrix given by
\[
 \hat P^\mu(x,y) = \hat P(x,y) + \hat P(x,0)\mu(y).
\]
That, is $Z^\mu_n$ evolves as a Markov chain with transitions given by $\hat P$ but
when it attempts to jump to zero, it comes back instantaneously to $\Lambda$
according to $\mu$. This is an ergodic Markov chain with invariant measure
$\Phi(\mu)$. If $\mu$ coincides with the QSD $\nu$ then it is invariant for $\hat P^\nu$. Also the reciprocal holds, i.e.: $\nu$ is
a QSD if and only if $\Phi(\nu)=\nu$. This is the key fact used in $\cite{FKMP}$
to prove their general existence result. The same fact is useful
here to build a method approximating the QSD. The main idea is to replace
$\nu$ in the ``return'' process (which is only used to define the dynamics at absorption times) by the empirical distribution of the history of
the process.

Let $M$ be the set of counting measures on $\Lambda$, and let $(V_n, \mu_n)$ be
the $\Lambda \times M-$valued Markov chain with transitions defined by
\be{aldous.proc}
 P(V_{n+1}=y, \mu_{n+1}= \mu + \delta_y|V_n=x, \, \mu_n=\mu) = \hat P(x,y) +
\hat P(x,0){\frac{\mu(y)}{|\mu|}},
\ee
where $|\mu|$ is the total mass of $\mu$.
In words, $V_n$ evolves as a Markov chain with transition matrix $\hat P$, but when
it attempts to jump to zero, it comes back to $\Lambda$ instantaneously
according to the empirical distribution of the history of the process $\mu$.

The following theorem was obtained by Aldous, Flannery and Palacios concerning this process.
\begin{theorem}[Aldous, Flannery and Palacios, \cite{AFP}] For the process
defined by  \eqref{aldous.proc} we have
\[
 \lim_{n\to\infty}\frac{1}{n} \mu_n = \nu.
\]
\end{theorem}
The proof follows by embedding $\mu_n$ in the skeleton of a supercritical
multitype branching process with offspring distribution given by
$ L(x,y)=$ number of visits to the state $y$, for the process with transition
matrix $P$, started at $x$. Details can be found in \cite{AFP}.

This method, with improvements to gain efficiency was implemented by in \cite{DdO2,DdO} to compute in particular
the conditioned evolution and QSD of the contact process in different graphs. No proof of convergence are available in this case.

\subsection{Supercritical multitype branching} 

The results of this section are the outcome of discussions with Pablo A.
Ferrari.

Consider the restricted matrix
$\tilde Q:= (q(x,y)\colon x, y \in \Lambda)$ and assume for the moment
that $\Lambda$ is finite. Let $Q_\alpha$ be the matrix
with entries
\[
 q_\alpha(x,y) = q(x,y) + \alpha \delta_{xy}, \qquad x, y \in \Lambda.
\]
Here $\delta_{xy}$ is Kronecker's delta function. We consider a continuous-time multitype branching process as described by
Athreya and Ney \cite[Chapter V.7]{AN}. The set of types is $\Lambda$. An
individual of type $x$ lives for a mean one exponential time and then splits
into a random offspring distribution $L({x})= (L(x,y), \, y \in \Lambda)$ on
$\N^\Lambda$ with means given by $\m_\alpha(x,y) = q_\alpha(x,y) +
\delta_{xy}$. We use $X(t)=(X(t,x), \, x \in \Lambda)$ for the number of
individuals of each type at time $t\ge 0$. See \cite{AN} for details on this
construction. It is well known (\cite[p. 202]{AN}) that the expected number of
individuals of type $y$ when we start the process with a unique individual of
type $x$ is given by $E_x(X(t,y))=e^{tQ_\alpha}(x,y)$ for all $x,y \in \Lambda$.

Since $\Lambda$ is finite, Perron-Frobenius theory (see also \cite{DS}) implies
the existence of a principal eigenvalue $\lambda_\alpha$ for $Q_\alpha$ (a real
eigenvalue with real part larger than all other eigenvalues) and a positive left
eigenvector $\nu$, normalized to be a probability measure in $\Lambda$. Observe
that $\lambda_\alpha$ is an eigenvalue of $Q_\alpha$ if and only if
$\lambda_\alpha=\lambda + \alpha$, with $\lambda$ eigenvalue of $\tilde
Q$ and hence $\lambda_\alpha$ is positive if $\alpha$ is large enough.
For those values of $\alpha$, the associated multitype branching process with
mean offspring matrix $\m_\alpha$ is supercritical and
the Kesten-Stigum theorem (\cite{AN}, see also \cite{GB} for a conceptual
proof) implies that
\[
\lim_{t \to \infty} \frac{X(t)}{|X(t)|} = \nu,
\]
almost surely on the event $\Omega_{{\rm surv}}:=\{|X(t)|>0 \text{ for all }
t>0 \}$. Now it is immediate that $\nu$ is a left eigenvector for $\tilde Q$
with eigenvalue $\lambda$. 
Since $\sum_{y\in \Lambda} q(x,y) = -q(x,0)$,
summing over $y$ on both sides of
\[
 \sum_{x\in\Lambda} \nu(x)q(x,y) = \lambda\nu(y),
\]
we get that $\lambda=-\sum_{x \in\Lambda}\nu(x)q(x,0)$ and hence $\nu$
verifies \eqref{intro-1}, which means that it is the unique QSD. This method is
useful as a practical tool to simulate QSDs. It is valid for finite $\Lambda$
and in any other case where a Kesten-Stigum type theorem is available. So, for a given
absorbing process with rates matrix $Q$, it is important to understand whether
the supercritical multitype branching process described above behaves as
predicted by the Kesten-Stigum theorem.

\section{Related open problems}
\label{relatedproblems}

The different methods for simulating QSDs mentioned previously still raise
many challenging open problems. 
We only mention here a few of them.

\paragraph{Rates of convergence to equilibrium and error bounds.}
Very little information has been obtained concerning the invariant measure of FV and more generally about the 
rate of convergence and error bounds for the approximation of the 
QSD, for each of the mentioned methods.
These would constitute very useful informations in particular for the FV
dynamics: how fast is the FV process converging to its equilibrium for fixed
$N$, and then how close is its empirical distribution from the QSD?

\paragraph{The other eigenvectors.}
It has been shown that the FV dynamics allow the estimation of the maximal
eigenvector of the infinitesimal generator associated with their driving
process in many situations.
One may naturally aims at showing similar results for
other eigenvectors. 

In \cite{BIHM}, the authors are interested in estimating eigenvectors of the
Laplacian on bounded domains. 
They  propose to adapt the dynamics of the
FV process introducing different types of particles and considering as driving
process random walks or Brownian motions. 
It would be of great interest to prove that this procedure 
converges in this case and can be adapted for other driving processes.

\paragraph{Convergence of FV driven by drifted random walk.}
As it was pointed out before, the main questions related to FV are the ergodicity
of the process for each $N$ and the convergence of its empirical
distribution. The case of FV driven by a drifted random
walk is remarkably still an open problem.
The difficulty stems from the fact that the driving process has an infinite
number of QSDs together with the fact that it is not $R$-positive
(see \cite{seneta-veres}).

\paragraph{The immortal particle for FV.}
Grigorescu and Kang \cite{GK} show that one can define the notion of immortal
particle as the trajectory from which all particles proceed when time tends to
infinity (in the context of diffusions processes).
The convergence of the trajectory of the immortal particle  towards a
stationary regime (when it exists) has not been yet investigated.

\paragraph{Relation with N-BBM}
As alluded in the introduction, Brunet, Derrida and coauthors
introduced a system where particles evolve driven by
one dimensional Brownian Motion and branch with a fixed rate, at which time the
left-most particle is instantaneously eliminated.
Very precise quantitative informations have been investigated for these
dynamics at various time scales \cite{berestycki,maillard}. 
We conjecture that long time behaviors of this system properly renormalized
are very similar to the one of the FV process.
A systematic comparison of the empirical distributions and proof of convergence
towards the minimal QSDs are open research avenues.

\paragraph{Number of jumps of FV driven by diffusions in compact time intervals}
Though not included in the context of countable state space, we would like to
mention the problem of the definition of the FV process for non-countable state
space.
 As already underlined in the introduction, and defined as an open
problem by K.
Burdzy \cite{BFP},
it is not known in general under which conditions the number of jumps of the FV
processes driven
by diffusions on finite or infinite domain does not accumulate 
in compact times intervals.

In \cite{bieniek} and \cite{GK}, some conditions on the domain were shown to be
sufficient to impede this phenomenon from happening. Also
\cite{BBS} partly deals with this problem.

\

 {\bf Acknowledgments.}
We would like to warmly thank Amine Asselah and Pablo Ferrari for our fruitful discussions and collaborations
on this subject. PG and MJ are partially supported by UBACyT 20020090100208, ANPCyT PICT No. 2008-0315 and CONICET PIP 2010-0142 and 2009-0613.

\bibliographystyle{plain}
\bibliography{biblio2}

\end{document}